\documentclass[11pt]{article}
\usepackage{amsmath, amsthm, amssymb, pdfsync, psfrag, verbatim,color}
\usepackage{hyperref}
\usepackage {mathrsfs, graphicx, newcent,wrapfig,multirow}
\usepackage[round]{natbib}
\usepackage{breakcites}
\usepackage[affil-it]{authblk}

\newcommand{\bbE}{\mathbb E}
\newcommand{\bbF}{\mathbb F}

\newcommand{\bbH}{\mathbb H}

\newcommand{\bbN}{\mathbb N}

\newcommand{\bbP}{\mathbb P}

\newcommand{\bbR}{\mathbb R}
\newcommand{\bbS}{\mathbb S}

\newcommand{\scA}{\mathcal A}
\newcommand{\scB}{\mathcal B}

\newcommand{\scE}{\mathcal E}
\newcommand{\scF}{\mathcal F}

\newcommand{\scL}{\mathcal L}

\newcommand{\veps}{\varepsilon}

\newcommand{\ang}[1]{\ensuremath{ \left \langle #1 \right \rangle }}
\newcommand{\norm}[1]{\ensuremath{\left\| #1 \right\|}}
\newcommand{\abs}[1]{\ensuremath{\left| #1 \right|}}

\DeclareMathOperator*{\esssup}{ess\,sup}

\newcommand{\indicator}[1]{\ensuremath{\mathbf{1}_{#1}}}

\newcommand{\half}{\frac{1}{2}}

\newcommand{\crl}[1]{\ensuremath{ \left\{ #1 \right\} }}
\newcommand{\edg}[1]{\ensuremath{ \left[ #1 \right] }}
\newcommand{\brak}[1]{\ensuremath{\left( #1 \right)}}

\newtheorem{theorem}{Theorem}[section]
\newtheorem{definition}[theorem]{Definition}
\newtheorem{proposition}[theorem]{Proposition}
\newtheorem{corollary}[theorem]{Corollary}
\newtheorem{lemma}[theorem]{Lemma}
\newtheorem{remark}[theorem]{Remark}
\newtheorem{example}[theorem]{Example}
\newtheorem{examples}[theorem]{Examples}
\newtheorem{foo}[theorem]{Remarks}
\newenvironment{Example}{\begin{example}\rm}{\end{example}}

\newenvironment{Remark}{\begin{remark}\rm}{\end{remark}}

\title{Coupled FBSDEs with Measurable Coefficients and\\ its Application to Parabolic PDEs}
\author[1]{Kihun Nam}
\author[2]{Yunxi Xu}
\affil[1]{Department of Mathematics, Monash University, Australia}
\affil[1]{Centre for Quantitative Finance and Investment Strategies, Monash University, Australia}
\affil[2]{Department of Mathematics, Monash University, Australia}
\date{\today}

\begin{document}
	\maketitle
	\begin{abstract}
		Using purely probabilistic methods, we prove the existence and the uniqueness of solutions fora system of coupled forward-backward stochastic differential equations (FBSDEs) with measurable, possibly discontinuous coefficients. As a corollary, we obtain the well-posedness of semilinear parabolic partial differential equations (PDEs)
\begin{align*}
			&\mathcal{L} u(t,x)+F(t,x,u,\partial_x u)=0;\qquad u(T,x)=h(x)\\
			&\mathcal{L}:=\partial_t+\half\sum_{i,j=1}^m(\sigma\sigma^\intercal)_{ij}(t,x)\partial^2_{x_ix_j}
\end{align*}
		in the natural domain of the second-order linear parabolic operator $\scL$. We allow $F$ and $h$ to be discontinuous with respect to $x$. Finally, we apply the result to optimal policy-making for pandemics and pricing of carbon emission financial derivatives.
		\\[2mm]
	\end{abstract}
	\setcounter{equation}{0}
	\section{Introduction}
	In this study, we prove the existence and uniqueness of a solution for the coupled forward-backward stochastic differential equations (FBSDEs)
	\begin{equation}\label{intro:fbsde}
		\begin{aligned}
			dX_t&=\tilde g(t,X_t,Y_t,Z_t)dt+\sigma(t,X_t) dW_t; &X_0&=x\in\bbR^m\\
			dY_t&=-f(t,X_t,Y_t,Z_t)dt+Z_tdW_t;&Y_T&=h(X_T)
		\end{aligned}
	\end{equation}
	where $W$ is an $n$-dimensional Brownian motion, the coefficients $\tilde g, \sigma, f,$ and $h$ are $\bbR^m$-, $\bbR^{m\times n}$-, $\bbR^d$-, and $\bbR^d$-valued deterministic measurable functions, respectively. In particular, we are interested in the case in which these coefficients are discontinuous.
	
	\cite{bismut1973conjugate} introduced FBSDE as a dual problem of stochastic control, which corresponds to the Pontryagin principle. \cite{pardoux1990adapted,pardoux1992backward} proved the existence and uniqueness of the solution and its relationship with partial differential equations (PDEs) when the coefficients are non-linear Lipschitz and the forward and backward equations are decoupled. The well-posedness of the coupled FBSDE was studied using various methods (see \cite{ma1999forward} and \cite{ma2015well} and references therein for a survey).
	
	Three methods are commonly used to solve coupled FBSDEs: 1) contraction mapping, 2) the four step scheme, and 3) the method of continuation. The contraction mapping method was first used by \cite{antonelli1993backward}, and then \cite{pardoux1999forward} detailed the method. This method works well when $T$ is sufficiently small, but difficult to apply for an arbitrarily large $T$. The four step scheme, introduced by \cite{ma1994solving}, uses the concept of ``decoupling field'' to find an FBSDE solution by solving a corresponding PDE. The idea of the decoupling field was extended to FBSDEs with random coefficients by \cite{ma2015well}. Finally, the method of continuation, which was developed in \cite{hu1995solution,yong1997finding, peng1999fully,yong2010forward}, can prove the well-posedness of FBSDEs with random coefficients for arbitrary time intervals under  ``monotonicity conditions'', which is often not easy to verify.
	
	As our FBSDE has deterministic coefficients, the most natural approach would be the four step scheme by assuming a decoupling field $u$ that satisfies $Y_t=u(t, X_t)$. If the coefficients are sufficiently smooth, then the corresponding PDE 
	\begin{equation}\label{intro:pde}
		\begin{aligned}
			&\scL:=	\partial_t+\half\sum_{i,j=1}^m(\sigma\sigma^\intercal)_{ij}(t,x)\partial^2_{x_ix_j}, \qquad \nabla:=(\partial_{x_1},\partial_{x_2},\cdots,\partial_{x_m}),\\
			&\scL u(t,x)+ \nabla u(t,x) \tilde g(t,x,u,\nabla u\sigma)+f(t,x,u,\nabla u\sigma)=0;\qquad u(T,x)=h(x)
		\end{aligned}
	\end{equation}
	would have a solution $u\in C^{1,2}([0,T]\times\bbR^m;\bbR^d)$ 
	and we only need to find a strong solution of SDE
	\begin{align*}
		dX_t=\tilde g(t, X_t,u(t,X_t),(\nabla u\sigma)(t,X_t))dt+\sigma(t,X_t)dW_t;\qquad X_0=x.
	\end{align*}
	However, when the coefficients are discontinuous in $x$, additional requirements are needed to obtain sufficient regularity of $u$. 
	
	Parabolic PDEs with measurable coefficients have been studied broadly in PDE literatures (see \cite{krylov1999analytic, maugeri2000elliptic,kim2007parabolic,krylov2007parabolic, kim2017lq} and references therein). Most of the previous literatures on PDE with discontinuous coefficients focused on the viscosity solution or the class of solutions in the Sobolev space. 
	The viscosity solution cannot be applied if $d>1$, because the comparison principle does not hold in general. On the other hand, the decoupling field $u$ does not have to be in the Sobolev space $W^{1,2}_p([0,T]\times\bbR^m)$; $u(t,X_t)$ only need to be an It\^o process. When $|b|+|\sigma\sigma^\intercal|$ is bounded and $\sigma\sigma^\intercal$ is a continuous function with uniform ellipticity, \cite{chitashvili1996generalized} and \cite{mania2001semimartingale} proved that $u\in{\rm Dom}(\scL)$ for an operator $\scL$ defined by the closure of $\partial_t+\half\sum_{i,j}(\sigma\sigma^\intercal)_{ij}\partial^2_{x_ix_j}$, is the necessary and sufficient condition for $Y_t=u(t,X_t)$ to be an It\^o process. Since $u\in{\rm Dom}(\scL)$ implies that $u$ is differentiable in $x$, the space ${\rm Dom}(\scL)$ is the natural choice to seek for a solution of \eqref{intro:pde}.
	
	This study proves the existence of a solution $(X,Y,Z)$ and a measurable function $u\in {\rm Dom}(\scL)$ such that $Y_t=u(t,X_t)$. The function $u$ necessarily satisfies \eqref{intro:pde}. In this sense, our method can be seen as a probabilistic generalization of the four step scheme for measurable coefficients. On the other hand, we can also view our article as a study of parabolic PDEs with measurable coefficients with a slightly weaker class of solutions than the Sobolev solutions.
	
	The FBSDE \eqref{intro:fbsde} with measurable discontinuous coefficients has been studied previously. In particular, when $\tilde g(t,x,y,z)$ does not depend on $(y,z)$, we only need to focus on the backward equation of \eqref{intro:fbsde} as the system decouples. \cite{el1997backward} proved the existence and uniqueness of a solution when $f(t,x,y,z)$ is Lipschitz with respect to $(y,z)$ and \cite{hamadene1997bsdes} proved the existence of a solution when $f(t,x,y,z)$ is continuous with respect to $(y,z)$. For fully coupled FBSDEs, \cite{carmona2013singular} proved the existence and uniqueness of a solution for a FBSDE appearing from carbon emission derivative market when $h$ is a step function. More general result is obtained by \cite{luo2020strong}, who used the mollification and PDE techniques to prove the existence of a solution for \eqref{intro:fbsde}. Their condition requires that $\sigma$ is a constant, $h$ is bounded, and either (i) $\tilde g$ and $f$ have linear growth on $y$, uniformly in $(t,x,z)$, or (ii) $h$ is Lipschitz and $\tilde g$ and $f$ have linear growth on $(y,z)$, uniformly in $(t,x)$. On the other hand, \cite{chen2018forward} studied the case where $\tilde g(t,x,y,z)$ is a step function in $y$.
	
	This study provides sufficient conditions for the existence and uniqueness of a strong solution for \eqref{intro:fbsde} with deterministic measurable coefficients, which can be discontinuous, when $\sigma(t,x)$ is uniformly non-degenerate. In particular, our results generalize \cite{luo2020strong} in terms of the growth of coefficients and nonconstant $\sigma$ using the simpler measure-change technique. In short, we decoupled the FBSDE \eqref{intro:fbsde} with the Girsanov transform, and then verified that it is actually a strong solution using the existence of the Markovian representation $(Y_t, Z_t)=(u(t,X_t),d(t,X_t))$. While the Markovian representation resembles the decoupling field in the four step scheme, which is defined through PDE, the existence of $u$ and $d$ stems from a purely probabilistic argument based on \cite{cinlar1980semimartingales}.
	
	This article consists of four sections. In Section 2, we present our main result, i.e., Theorem \ref{mainthm} and provide a corollary stating the relationship with PDE. In Section 3, we present two applications of this study. In Section 3.1, we study the optimal control of the spread of an infectious disease when there are various degrees of medical service capacities, such as the number of hospital beds and medicine supply. In Section 3.2, we present another application, the pricing of carbon emission allowance derivatives. In Section 4, we prove Theorem \ref{mainthm}.
	\section{Existence and Uniqueness of Solution}
	Let $(\Omega,\scF,\bbP)$ be a filtered probability space with an $n$-dimensional Brownian motion $W$ and its augmented filtration $\bbF$ generated by $W$. For a matrix $A$, we denote $\abs{A}:=\sqrt{{\rm tr}(A A^\intercal)}$ and consider a vector as a column matrix. As usual, we assume Borel sigma-algebra on Euclidean spaces. 
	Let $\bbS^p(E)$ and $\bbH^p(E)$ for a Euclidean space $E$ be the space of the adapted $E$-valued processes $X$ with
	\begin{align*}
		\norm{X}_{\bbS^p}^p:=\bbE\sup_{t\in[0,T]}|X_t|^p\text{ and } \norm{X}_{\bbH^p}^p:=\bbE\edg{\int_0^T|X_t|^2dt}^{p/2}.
	\end{align*}
For a function $u\in C^{0,1}\brak{[0,T]\times \bbR^m;\bbR^d}$, we denote $\nabla u:=\brak{\partial_{x_1} u, \partial_{x_2} u,\cdots, \partial_{x_m} u}$.
	Let
	\begin{align*}
		b&:[0,T]\times\bbR^m\to\bbR^{m}\\
		\sigma&:[0,T]\times\bbR^m\to\bbR^{m\times n}\\
		f&:[0,T]\times\bbR^m\times\bbR^d\times\bbR^{d\times n}\to\bbR^d\\
		g&:[0,T]\times\bbR^m\times\bbR^d\times\bbR^{d\times n}\to\bbR^n\\
		h&:\bbR^m\to\bbR^d.
	\end{align*}
	be (jointly) measurable functions. 
	Unless otherwise stated, we assume the following conditions:
	\begin{itemize}
		\item $\sigma$ is uniformly nondegenerate, that is, there exists a constant $\veps>0$ such that
		\[
		\veps^{-1}|x'|^2\leq (x')^\intercal(\sigma\sigma^\intercal)(t,x) x'\leq \veps|x'|^2\\
		\]
		for all $x'\in\bbR^m$ and $(t,x)\in[0,T]\times\bbR^m$.
		\item There exists a positive constant $\kappa$ such that, 
		$$
		|b(t,0)|+\sup_{|x-x'|\leq 1}|b(t,x)-b(t,x')|\leq \kappa
		$$
		for all $t\in[0,T], x,x'\in\bbR^m$.
	\end{itemize}
	
	\begin{Remark}
		Under the non-degeneracy assumption on $\sigma$, we have
		\begin{align*}
			b(t,x)+\sigma(t,x)g(t,x,y,z)
			&=\tilde b(t,x)+\sigma(t,x)\tilde g(t,x,y,z),
		\end{align*}
		where $\tilde b=b-k$ and $\tilde g=\sigma^\intercal(\sigma\sigma^\intercal)^{-1}k+g$. This adds flexibility to the conditions described below.
	\end{Remark}
	
	\begin{Remark}\label{lineargrowth}
		Note that $b(t,x)$ can exhibit linear growth in $x$. For example, let $b(t,x)=b_0(t,x)+b_1(t,x)$ where $\sup_{t\in[0,T]}|b_0(t,0)+b_1(t,0)|=1$, $b_0(t,x)$ is H\"older continuous in $x$, and $b_1(t,x)$ is bounded. 
	\end{Remark}
	
	We use the following short-hand notations for different conditions on the coefficients, where $
	\bar f(t,x,y,z):=f(t,x,y,z)+zg(t,x,y,z)$, $C$ and $r$ are nonnegative constants, $\theta:\bbR\to\bbR_+$ is a strictly increasing function, and $\rho_r:\bbR_+\to\bbR_+$ is a nondecreasing function with $\rho_r\equiv 0$ for $r> 0$:
	\begin{itemize}
		
		\item[{\rm (F1)}] $|b(t,x)|\leq C$ and $\sigma(t,x)$ is locally Lipschitz with respect to $x$.
		\item[{\rm (F2)}] $|b(t,x)|\leq C$, $m=n=1$, and either
		\begin{itemize}
			\item[(i)] \(\int \frac{du}{\theta(u)}=\infty\) and
			\(
			\abs{\sigma(t,x)-\sigma(t,y)}^2\leq \theta(|x-y|),
			\)
			or
			\item[(ii)]  $\theta$ is bounded and \(
			\abs{\sigma(t,x)-\sigma(t,y)}^2\leq \abs{\theta(x)-\theta(y)}.
			\)
		\end{itemize}
		\item[{\rm (F3)}] $\sigma(t,x)$ is a constant matrix.
		
		\item[{\rm (B1)}] $|h(x)|\leq C(1+|x|^r)$, $\bar f(t,x,y,z)$ is continuous in $(y,z)$, and 
		\begin{equation*}
			\begin{aligned}
				|f(t,x,y,z)|&\leq C(1+|x|^r+|y|+|z|)\\
				|g(t,x,y,z)|&\leq C(1+\rho_r(|y|)).
			\end{aligned}
		\end{equation*}
		
		\item[{\rm (B2)}] $|h(x)|\leq C$, $\bar f^i(t,x,y,z)=\tilde f^i(t,x,z^i)+\hat f^i(t,x,y,z)$ such that 
		\begin{equation*}
			\begin{aligned}
				|\hat f(t,x,y,z)|&\leq C(1+|y|)\\
				|\tilde f(t,x,z)|&\leq C|z|^2\\
				|g(t,x,y,z)|&\leq C(1+\rho_r(|y|))\\
				\abs{\hat f(s,x,y,z)-\hat f(s,x,y',z')}&\leq C(|y-y'|+|z-z'|),\forall y,y'\in\bbR^d, z,z'\in\bbR^{d\times n}\\
				\abs{\tilde f(s,x,z_1)-\tilde f(s,x,z_2)}&\leq C(1+|z_1|+|z_2|)|z_1-z_2|,\forall z_1,z_2\in\bbR^{d\times n}.
			\end{aligned}
		\end{equation*}
		\item[{\rm (B3)}] $d=1$, $|h(x)|\leq C$, $\bar f(t,x,y,z)$ is continuous with respect to $(y,z)$, and 
		\begin{equation*}
			\begin{aligned}
				|f(t,x,y,z)|&\leq C(1+|y|+|z|^2)\\
				|g(t,x,y,z)|&\leq C(1+\rho_r(|y|)).
			\end{aligned}
		\end{equation*}
		\item[{\rm (B4)}] $|h(x)|\leq C(1+|x|^r)$, $\bar f(t,x,y,z)$ is continuous in $(y,z)$, and \begin{equation*}
			\begin{aligned}
				|f^i(t,x,y,z)|&\leq C(1+|x|^r+|y^i|) \text{ for all }i=1,2,..., d\\
				|g(t,x,y,z)|&\leq C(1+|x|+\rho_r(|y|)).
			\end{aligned}
		\end{equation*}
		\item[{\rm (U1)}] $\bar f(t,x,y,z)$ is globally Lipschitz continuous with respect to $(y,z)$, or
		\item[{\rm (U2)}] $d=1$, $|h(x)|\leq C$, $\bar f(t,x,y,z)$ is differentiable with respect to $(y,z)$, and for any $M,\veps>0$, there exist $l_M,l_\veps\in L^1([0,T];\bbR_+), k_M\in L^2([0,T];\bbR_+)$, and $C_M>0$ such that $\bar f$ satisfies
		\begin{align*}
			\abs{\bar f(t,x,y,z)}&\leq l_M(t)+C_M |z|^2\\
			\abs{\partial_z\bar f(t,x,y,z)}&\leq k_M(t)+C_M|z|\\
			\abs{\partial_y\bar f(t,x,y,z)}&\leq l_\veps(t)+\veps|z|^2
		\end{align*}
		for all $(t,x,y,z)\in[0,T]\times\bbR^m\times[-M,M]\times\bbR^{1\times n}$.
	\end{itemize}
	
	\begin{theorem}\label{mainthm}
		Assume that there exist nonnegative constants $C, r>0$, a strictly increasing function $\theta:\bbR\to\bbR_+$, and a nondecreasing function $\rho_r:\bbR_+\to\bbR_+$ with $\rho_r\equiv 0$ for $r>0$ that satisfies either of the following conditions: 
		\begin{itemize}
			\item one of {\rm (F1), (F2), (F3)} and one of {\rm (B1), (B2), (B3)} hold for any $(t,x,y,z)\in[0,T]\times\bbR^m\times\bbR^d\times\bbR^{d\times n}$.
			\item {\rm (F3)} and {\rm (B4)} hold for any $(t,x,y,z)\in[0,T]\times\bbR^m\times\bbR^d\times\bbR^{d\times n}$.
		\end{itemize}
		Then, FBSDE 
		\begin{equation}\label{main:fbsde}
			\begin{aligned}
				dX_t&=\brak{b(t,X_t)+\sigma(t,X_t) g(t,X_t,Y_t,Z_t)}dt+\sigma(t,X_t) dW_t; &X_0&=x\\
				dY_t&=-f(t,X_t,Y_t,Z_t)dt+Z_tdW_t;&Y_T&=h(X_T)
			\end{aligned}
		\end{equation}
		has a strong solution in $\bbH^2(\bbR^m)\times\bbH^2(\bbR^d)\times\bbH^{2}(\bbR^{d\times n})$.
		In particular, if $r=0$, then the FBSDE has a strong solution $(X,Y,Z)$ such that $\esssup_{\omega\in\Omega}\sup_{t\in[0,T]}|Y_t(\omega)|<\infty$. In addition, if either {\rm (U1), (U2),} or {\rm (B2)} holds, then the solution is unique.
	\end{theorem}
	\begin{proof}
		The proof is given in Section 4.
	\end{proof}
	Under the assumptions of Theorem \ref{mainthm}, we have measurable functions $u:[0,T]\times\bbR^m\to\bbR^d$ and $d:[0,T]\times\bbR^m\to\bbR^{d\times n}$ such that $Y_t=u(t,X_t),Z_t=d(t,X_t)$ for almost every $(t,\omega)\in[0,T]\times\Omega$. To state that $u$ is a solution of a parabolic PDE, let us define $V^{\scL}_\mu(loc)$, the class of $\scL$-differentiable functions.
	\begin{definition}[\cite{chitashvili1996generalized}]
		Let $\mu(ds,dy):=p(0,x,s,y)dsdy$, where $p$ is the  transition density corresponding to SDE $dX_t=\sigma(t,X_t)dW_t$. Further, for a function $f\in C^{1,2}$, we define
		\begin{align*}
			\scL f&:=\partial_tf+\half\sum_{i,j=1}^m(\sigma\sigma^\intercal)_{ij}(t,x)\partial^2_{x_ix_j}f.
		\end{align*}
		We say $u$ belongs to $V^{\scL}_\mu(loc)$, if there exists a sequence of functions $(u_n)_{n\geq 1}\subset C^{1,2}$, a sequence of bounded measurable domains $D_1\subset D_2\subset\cdots$ with $(0,x)\in D_1$ and $\cup_{n\in\bbN} D_n=[0,T]\times\bbR^m$, and a measurable locally $\mu$-integrable function $\scL u$ such that
		\begin{itemize}
			\item $\tau_k:=\crl{t>0:(t,X_t)\notin D_n}$ are stopping times with $\tau_n\nearrow T$.
			\item For each $k\geq 1$,
			\begin{align*}
				\sup_{s\leq\tau_k}\abs{u_n(s,X_s)-u(s,X_s)}\xrightarrow{n\to\infty}0 \qquad \text{a.s.}\\
				\iint_{D_k}\abs{\scL u_n(s,X_s)-\scL u(s,X_s)}\mu(ds,dx)\xrightarrow{n\to\infty}0.
			\end{align*}
		\end{itemize}
		Then, we define the $\scL$-derivative of $u$ by $\scL u$. 	Moreover, if $u\in V^{\scL}_\mu(loc)$, then there exists $\nabla u(t,x)$ such that
		\begin{align*}
			\iint_{D_k}\abs{\nabla u_n(s,X_s)-\nabla u(s,X_s)}^2\mu(ds,dx)\xrightarrow{n\to\infty}0.
		\end{align*}
		We define $\nabla u$ to be the generalized gradient of $u$.
	\end{definition}
	For an open set $D\subset\bbR^{1+m}$, let $W^{1,2}_p(D)$ be the
	completion of $C^{1,2}((0,T)\times D)$ with respect to the norm
	\[
	\norm{u}:=\sup_{(t,x)\in\bar D}|u(t,x)|+\norm{\partial_t u}_{L^p}+\norm{\nabla u}_{L^p}+\norm{\partial_{xx}^2 u}_{L^p}.
	\]
	We define $W^2_p(D)$ similarly.	Note that $W^{1,2}_p(D)$ is equivalent to the usual Sobolev space for continuous $u$ if $D$ has smooth boundary and $p\geq m+1$: see p47 of \cite{krylov1980controlled}.  
	\begin{Remark}[Corollary 1 of \cite{chitashvili1996generalized}]
		For $p\geq m+1$,
		\[W^{1,2}_p(D)\subset V^\scL_\mu(loc)\]
		for any bounded measurable domain $D\subset\bbR^m$.
	\end{Remark}
	
	We have the following corollary.
	\begin{corollary}\label{maincor}
		Assume the existence conditions in Theorem \ref{mainthm}. In addition, we assume that $\sigma\sigma^\intercal$ is continuous. Then, there exists $u\in V^{\scL}_\mu(loc)$ that satisfies
		\begin{equation}\label{pde}
			\begin{aligned}
				\scL u(t,x)&+ \nabla u(t,x) \brak{b(t,x)+\sigma(t,x) g(t,x,u,\nabla u\sigma)}+f(t,x,u,\nabla u\sigma)=0;\qquad u(T,x)=h(x).
			\end{aligned}
		\end{equation}
		If the uniqueness condition in Theorem \ref{mainthm} holds as well, then there is a unique $u\in V^{\scL}_\mu(loc)$ satisfying \eqref{pde}.
		Additionally, if assume the conditions
		\begin{equation}\label{regpde}
			\begin{aligned}
				\bullet&\qquad h\in W_p^2(\bbR^m),\\
				\bullet&\qquad b(t,x) \text{ is uniformly bounded in } (t,x)\\
				\bullet&\qquad (\sigma\sigma^\intercal)(t,x) \text{ is uniformly continuous with respect to $x$ for each }t\in[0,T],\text{ and }\\
				\bullet&\qquad \int_{[0,T]\times\bbR^{m}}\sup_{z\in\bbR^{d\times n}}\brak{|f(t,x,y,z)|^p}dtdx\leq C(1+\rho_r(|y|))\text{ and } |g(t,x,y,z)|\leq C(1+\rho_r(|y|)),
			\end{aligned}
		\end{equation}
		then $u\in W^{1,2}_p([0,T)\times \bbR^m)$.
	\end{corollary}
	\begin{proof}
		The existence of $u\in V^{\scL}_\mu(loc)$ satisfying \eqref{pde} is an immediate consequence of Theorem 1 in \cite{chitashvili1996generalized}.
		
		Assume conditions \eqref{regpde}. When $r=0$, conditions (B1)--(B4) implies the boundedness of $u$ by the same argument in the proof of Proposition \ref{B1H23}. On the other hand, if $r>0$, then $\rho_r\equiv 0$.  Therefore, without loss of generality, we can assume that $|g(t,x,y,z)|\leq C$ and
		\[
		\norm{f(\cdot,\cdot, u,\nabla u \sigma)}_{L^p}\leq \brak{\int_{[0,T]\times \bbR^m}\sup_{z\in\bbR^{d\times n}}\brak{|f(t,x,y,z)|^p}dtdx}^{1/p}\leq C.
		\]
		As we have a measurable function $u$ and $\nabla u$, let us define
		\begin{align*}
			L&:=\scL+\sum_{i=1}^m(b+\sigma g)^i\partial_{x_i}\\
			F(t,x)&:=f(t,x,u(t,x),(\nabla u \sigma)(t,x))-L h(x).
		\end{align*}
		Then, $\tilde u(t,x):=u(t,x)-h(x)$ solves the following PDE
		\begin{align*}
			L \tilde u=F;\qquad \tilde u(T,x)=0.
		\end{align*}
		Note that the above PDE is linear parabolic with measurable $F\in L^p([0,T]\times \bbR^m)$, $b+\sigma g$ is bounded, and  $(\sigma\sigma^\intercal)(t,\cdot)\in VMO(\bbR^m)$. Therefore, it satisfies the condition of Theorem 2.1 of \cite{krylov2007parabolic}, and $\tilde u\in W^{1,2}_p([0,T)\times \bbR^m)$ is a unique solution. As a result, $u=\tilde u+h$ is the unique solution of \eqref{pde} and $u\in W^{1,2}_p([0,T)\times \bbR^m)$.
	\end{proof}
	\section{Applications}
	In this section, we provide simple applications of our main result to the optimal control of the spread of an infectious disease and the carbon emission allowance market.
	\subsection{Controlling the Spread of an Infectious Disease }
	Let $W$ be a one-dimensional Brownian motion, $P$ be the number of infections, and $\alpha$ be the measures imposed by the policymaker to stop the spread. The admissible set for $\alpha$ is the set of non-negative adapted processes in $\bbH^2(\bbR)$. For measurable functions $\theta:\bbR\to\bbR$ and positive constant $\sigma$, assume that $P$ follows the dynamics
	\begin{align}\label{Pdynamics}
		\frac{dP_t}{P_t}=\brak{\theta(\log P_t)+\half|\sigma|^2 -\alpha_t}dt+\sigma dW_t;\qquad P_0=e^x.	
	\end{align}
	
	The interpretation of the dynamics is straightforward. Assuming there is no randomness in the spread ($\sigma= 0$), $\theta(\log P)$ represents the exponent of the infection growth when there is no intervention ($\alpha=0$). Function $p$ represents the compliance of citizens to the policy. In summary, if policy $\alpha$ is introduced, the growth exponent will be reduced to $\theta(\log P)-\alpha$.
	
	Let us define $X=\log P$. By It\^o formula, \eqref{Pdynamics} transforms to
	\begin{align*}
		dX_t&=\brak{\theta(X_t)- \alpha_t}dt+\sigma dW_t; &X_0&=x.
	\end{align*}
	Our objective as the policymaker is to minimize
	\begin{align*}
		J(\alpha):=\bbE\edg{\int_0^T|\alpha_t|^2+q(X_t)dt}.
	\end{align*}
	Here, $|\alpha_t|^2$ represents the running cost of the policy $\alpha$, and $q:\bbR\to [0,\infty)$ is the cost incurred by the number of infections. 
	\begin{Remark}
		It is realistic to assume that $q$ is a non-differentiable function, as it is the cost of infection. For example, consider that there is a capacity for medical services. If the infected patient number $P$ exceeds a certain level, there will be a shortage of medical services, which will cost a lot more per additional patient. 
	\end{Remark}
	
	Let the corresponding Hamiltonian and its minimizer be
	\begin{align*}
		H(t,x,y,z,\pi)&:=(\theta(x)-\pi)y+|\pi|^2+q(x)\\
		\arg\min_\pi H(t,x,y,z, \pi)&=\frac{(y\vee 0)}{2}.
	\end{align*}
	The following proposition is a version of Theorem 4.25 \cite{carmona2016lectures} for convex, possibly not continuously differentiable $\theta$ and $q$. Here, $\partial_{+}$ denotes the right derivative.
	\begin{proposition}\label{epicarmona}
		Assume that $\theta$ and $q$ are convex Lipschitz functions, $q$ is non-decreasing, and $\sigma>0$. 
		Let $(X,Y,Z)\in\bbH^2(\bbR)\times\bbH^2(\bbR)\times\bbH^{2}(\bbR)$ be the unique solution of FBSDE
		\begin{equation}\label{epiFBSDE}
			\begin{aligned}
				dX_t&=\brak{\theta(X_t)-\frac{(Y_t\vee 0)}{2}}dt+\sigma dW_t; &X_0&=x\\
				dY_t&=-\brak{\partial_{+}q(X_t)+\partial_{+}\theta(X_t)Y_t}dt+Z_tdW_t;&Y_T&=0
			\end{aligned}
		\end{equation}
		such that $Y$ is bounded.
		Then, for $\alpha^*_t:=(Y_t\vee 0)/2$, we have $J(\alpha^*)\leq J(\alpha)$ for any non-negative process $\alpha\in\bbH^2(\bbR)$.
	\end{proposition}
	Before we prove the proposition, we need the following lemma.
	\begin{lemma}The following holds:
		\begin{itemize}
			\item[(i)] \eqref{epiFBSDE} has a unique solution $(X,Y,Z)\in\bbH^2(\bbR)\times\bbH^2(\bbR)\times\bbH^{2}(\bbR)$ such that $Y$ is bounded almost surely.
			\item[(ii)] There exists a constant $C$ such that $Y_t\in[0,C]$ for all $t$ almost surely.
			\item[(iii)] $\bbE\int_0^T (X_t-X^\alpha_t)Z_tdW_t=0$ for any $\alpha\in\bbH^2(\bbR)$, where
			\begin{align*}
				dX^\alpha_t&=\brak{\theta(X^\alpha_t)-\alpha_t}dt+\sigma dW_t;&X_0&=x.
			\end{align*}
		\end{itemize}
	\end{lemma}
	\begin{proof}
		First, we prove (ii) under the assumption that \eqref{epiFBSDE} has a solution $(X,Y,Z)\in\bbH^2(\bbR)\times\bbH^2(\bbR)\times\bbH^{2}(\bbR)$.	Note that there exists a constant $C$ such that $\partial_{+}q\in[0,C]$ and $\partial_{+}\theta\in[-C,C]$. By the comparison principle, $Y^d_t\leq Y_t\leq Y^u_t$ for all $t\in[0,T]$ almost surely, where
		\begin{align*}
			dY^u_t&=-\brak{C+C|Y^u_t|}dt+Z^u_tdW_t;&Y^u_T&=0\\
			dY^d_t&=C|Y^d_t|dt+Z^d_tdW_t;&Y^d_T&=0.
		\end{align*}
		As $Y^u_t=e^{C(T-t)}-1\leq e^{CT}$ and $Y^d\equiv0$ almost surely, the claim is proved.
		
		Next, we prove (i) by using the localization argument. For $C_y$, the bound of $Y$ we obtained in (ii), we define $\varphi$ be a smooth function on $\bbR$ satisfying 
		\[
		\varphi(y)= 
		\begin{cases}
			y,& \text{if } y\in[0,C_y]\\
			0,              & \text{if } y\in(-\infty,-1]\cup[C_y+1,\infty)
		\end{cases}
		\]
		and $|\varphi(y)|\leq |y|$.
		Consider the FBSDE
		\begin{equation}\label{localepi}
			\begin{aligned}
				d\tilde X_t&=\brak{\theta(\tilde X_t)-\frac{\varphi(\tilde Y_t)}{2}}dt+\sigma dW_t; &X_0&=x\\
				d\tilde Y_t&=-\brak{\partial_{+}q(\tilde X_t)+\partial_{+}\theta(\tilde X_t)\varphi(\tilde Y_t)}dt+\tilde Z_tdW_t;&Y_T&=0.
			\end{aligned}
		\end{equation}
		Let $b(t,x)\equiv \theta(x), \sigma(t,x):\equiv\sigma, g(t,x,y,z):=-\frac{\varphi(y)}{2}, f(t,x,y,z):=\partial_{+}q(x)+\partial_{+}\theta(x)\varphi(\tilde y)$, and $h\equiv 0$. As $\theta$ and $q$ are Lipschitz, $\partial_{+}\theta$ and $\partial_{+}q$ are bounded. Then, one can verify that the coefficients satisfy (F3), (B4), and (U2) with $r=0$. Therefore, there exists a unique $(\tilde X, \tilde Y,\tilde Z)\in\bbH^2(\bbR)\times\bbH^2(\bbR)\times\bbH^{2}(\bbR)$ such that $\tilde Y$ is bounded. In particular, by the same comparison arguement we used in the proof of (ii), we obtain $\tilde Y\in[0,C_y]$. Therefore, $(\tilde X, \tilde Y,\tilde Z)$ also solves \eqref{epiFBSDE}. Therefore, we proved the existence of a solution.
		On the other hand, let $(X',Y',Z')\in\bbH^2(\bbR)\times\bbH^2(\bbR)\times\bbH^{2}(\bbR)$ be another solution to \eqref{epiFBSDE}. Then, by part (ii), $Y'_t\in[0,C]$ for all $t$. As $(X',Y',Z')$ also solves \eqref{localepi}, we have $(X',Y',Z')=(\tilde X, \tilde Y,\tilde Z)=(X,Y,Z)$. This proves the uniqueness of a solution.
		\\
		Now, let us prove (iii). As $\theta$ grows linearly, there exists $C>0$ such that $|\theta(x)|\leq C(1+|x|)$. Therefore,
		\begin{align*}
			|X_t^\alpha|&\leq |x|+\int_0^tC(1+|X^\alpha_s|)+|\alpha_s| ds+|\sigma W_t|\\
			&\leq |x|+CT+\int_0^T|\alpha_s| ds+C\int_0^t|X^\alpha_s|ds+|\sigma W_t|.
		\end{align*}
		Note that if we let $X^*_t:=\sup_{0\leq u\leq t} |X^\alpha_u|$,  there exists another constant $C'$ such that
		\begin{align*}
			\bbE|X^*_t|^2\leq C'\brak{1+\int_0^t\bbE|X^*_s|^2ds}
		\end{align*}
		as $\bbE\sup_{0\leq u\leq T}|W_u|^2<\infty$ and $\bbE\int_0^T|\alpha_s|^2ds<\infty$.
		By Gr\"onwall's inequality, $\bbE |X^*_T|^2=\bbE\sup_{0\leq u\leq T}|X_u^\alpha|^2<\infty$. We obtain $\bbE\sup_{0\leq u\leq T}|X_u|^2<\infty$ by the same argument.
		
		As
		\begin{align*}
			\bbE\sqrt{\ang{\int_0^\cdot (X_t-X^\alpha_t)Z_tdW_t}_T}&=\bbE\sqrt{\int_0^T |X_t-X^\alpha_t|^2|Z_t|^2dt}\\
			&\leq \bbE\sup_{0\leq u\leq T}|X_u-X_u^\alpha|\sqrt{\int_0^T|Z_t|^2dt}\\
			&\leq \half\bbE\sup_{0\leq u\leq T}|X_u-X_u^\alpha|^2+\half\bbE\int_0^T|Z_t|^2dt<\infty,
		\end{align*}
		 we prove the claim by the Burkholder-Davis-Gundy inequality.
	\end{proof}
	\begin{proof}[Proof of Proposition \ref{epicarmona}]
		For a given control $\alpha$, let us denote the corresponding dynamics of the log of infection number as $X'$, that is,
		\begin{align*}
			dX'_t&=\brak{\theta(X'_t)-\alpha_t}dt+\sigma dW_t;&X_0&=x.
		\end{align*}
		Note that, for $\alpha^*:=(Y\vee 0)/2$,
		\begin{align*}
			&J(\alpha^*)-J(\alpha)=\bbE\int_0^T\edg{|\alpha^*_t|^2-|\alpha_t|^2+q(X_t)-q(X'_t)}dt\\
			&=\bbE\int_0^T\edg{H(t,X_t, Y_t,Z_t,\alpha^*_t)-H(t,X'_t, Y_t,Z_t,\alpha_t)}dt - \bbE\int_0^T\edg{\theta(X_t)-\theta(X'_t)-\alpha^*_t+\alpha_t}Y_tdt.
		\end{align*}
		Note that, by integration by parts, we have
		\begin{align*}
			\bbE\int_0^T\edg{\theta(X_t)-\theta(X'_t)-\alpha^*_t+\alpha_t}Y_tdt
			&=\bbE\int_0^T(X_t-X'_t)\brak{\partial_{+}q(X_t)+\partial_{+}\theta(X_t)Y_t}dt
		\end{align*}
		Therefore,
		\begin{align*}
			&J(\alpha^*)-J(\alpha)\\
			&=\bbE\int_0^T\edg{H(t,X_t, Y_t,Z_t,\alpha^*_t)-H(t,X'_t, Y_t,Z_t,\alpha_t)}-(X_t-X'_t)\brak{\partial_{+}q(X_t)+\partial_{+}\theta(X_t)Y_t}dt\\
			&\leq \bbE\int_0^T\edg{H(t,X_t, Y_t,Z_t,\alpha_t)-H(t,X'_t, Y_t,Z_t,\alpha_t)}-(X_t-X'_t)\brak{\partial_{+}q(X_t)+\partial_{+}\theta(X_t)Y_t}dt.
		\end{align*}
		Here, we used the fact that
		\[
		H(t,X_t, Y_t,Z_t,\alpha^*_t)\leq H(t,X_t, Y_t,Z_t,\alpha_t)
		\]
		for any non-negative process $\alpha\in\bbH^2(\bbR)$. As $Y_t\geq 0$, $q(x)-q(x')\leq (x-x')\partial_{+}q(x)$ and $\theta(x)-\theta(x')\leq (x-x')\partial_{+}\theta(x)$, we have
		\begin{align*}
			H(t,X_t, Y_t,Z_t,\alpha_t)-H(t,X'_t, Y_t,Z_t,\alpha_t)&=(\theta(X_t)-\theta(X_t))Y_t+q(X_t)-q(X'_t)\\
			&\leq
			(X_t-X'_t)\brak{\partial_{+}q(X_t)+\partial_{+}\theta(X_t)Y_t}.
		\end{align*}
		Therefore, $J(\alpha^*)\leq J(\alpha)$.
	\end{proof}
	\subsection{Electricity Market with Carbon Emission Allowance}
	Let us provide an example of an FBSDE with measurable coefficients in the pricing of carbon emission allowance in the electricity market. We follow the example in \cite{carmona2013singular} except that we assume there exists a cost $c^i$ depending on the carbon emission abatement and the total cumulative carbon emission, whereas \cite{carmona2013singular} assumes only the cost's dependency on the carbon emission abatement. Heuristically, if total cumulative carbon emission increases, the government will try to reduce the marginal cost of carbon emission abatement (e.g. Emission Reduction Fund) and and the society will be urge to develop cost efficient green technologies.
	
	For simplicity, let $\bbP$ be a risk-neutral measure and let the cumulative emission of the $i$th firm ($i=1,2,..., N$) up to time $t$ be $E^i_t$. Assume that $E^i$s follow the dynamics
	\[
	E^i_t=E^i_0+\int_0^t\left(b^i\brak{s, \bar E_s}-\xi^i_s\right)ds+\int_0^t\sigma^i(s,\bar E_s)dW_s.
	\]
	where $\bar E_s:=\sum_{j=1}^NE^j_s$. Here, $b^i$ denotes the so-called \emph{business-as-usual}, the rate of emission without carbon regulation. The process $\xi^i$ is the instantaneous rate of abatement chosen by the firm. The firm controls its own abatement schedule $\xi^i$ and the carbon emission allowance quantity $\theta^i$, which is traded in the allowance market. Both control processes need to be $dt\otimes d\bbP$-square-integrable adapted processes, which is denote by $\scA$. The firm's wealth is given by
	\begin{align*}
		X^i_T\brak{\xi^i,\theta^i}=x^i+\int_0^T \theta^i_sdY_s-\int_0^Tc^i(\xi^i_s, \bar E_s)ds-E^i_TY_T
	\end{align*}
	where $x^i$ is the initial wealth, $Y$ is the allowance price, $c^i:\bbR^2\to\bbR$ is the cost occurred by the abatement $\xi^i$. We assume that $c^i(e,y)$ is jointly measurable and convex in $x$. Then, one can define
	\[
	g^i(e,y)=\arg\min_x\brak{c^i(x,e)-yx}.
	\]
	We assume that the utility of each firm is given by an increasing, strictly concave function $U:\bbR\rightarrow\bbR$, which satisfies the Inada conditions: $U'(-\infty)=+\infty\ \ \text{and}\ \ U'(+\infty)=0.$
	
	The corresponding optimization problem is to
	find a pair of $(\xi^i,\theta^i)\in\mathcal{A}$ that maximizes
	$\bbE U(X^i_T\brak{\xi^i,\theta^i}).$
	
	By the same argument in Proposition 1 of \cite{carmona2013singular}, we can deduce that $\xi^i_t=g^i(\bar E_t,Y_t)$ is the optimal control. Therefore,
	$(\bar E, Y)$ should solve the following FBSDE: for $b=\sum_{i=1}^N b^i, \sigma=\sum_{i=1}^N \sigma^i,$ and $g=\sum_{i=1}^N g^i,$
	\begin{equation}\label{fbsde:carbonex}
		\begin{aligned}
			d\bar E_t&=\edg{b(t,\bar E_t)-g(\bar E_t, Y_t)}dt+\sigma(t,\bar E_t)dW_t;&\bar E_0&\in\bbR\\
			dY_t&=Z_tdW_t;&Y_T&=\lambda 1_{[\Lambda,\infty)}(\bar E_T).
		\end{aligned}
	\end{equation}
	
	Here, the terminal condition of allowance is assumed to be an indicator function based on \cite{carmona2010market}. As one can see from the following example, our main theorem \ref{mainthm} generalizes Theorem 1 of \cite{carmona2013singular} as we allow all the coefficients to be discontinuous.
	
	\begin{Example}
		Assume that there exists constants $C, K$ and $\alpha^i\in(0,1), i=1,2,..., N$ such that, for all $(t,e,y)\in[0,T]\times\bbR\times\bbR$,
		\begin{align*}
			|b(t,e)|&\leq C\text{ and } C^{-1}\leq (\sigma\sigma^\intercal)(t,e)\leq C\\
			c^i(x,e)&=\half x^2\brak{1-\alpha^i\indicator{e\geq K}}.
		\end{align*}
	Furthermore, assume that $\sigma(t,x)$ is locally Lipschitz with respect to $x$.
	Then, \eqref{fbsde:carbonex} has a unique strong solution such that $Y$ is bounded.
	\end{Example}
\begin{proof}
	Note that
	\[
	g(e,y)=\sum_{i=1}^N\frac{y}{1-\alpha^i\indicator{e\geq K}}.
	\]
	Then it is easy to check (F1) and (B1) are satisfied with $r=0$. Therefore, there exists a strong solution of \eqref{fbsde:carbonex} such that $Y$ is bounded. Moreover, (U2) holds. The uniqueness result of Theorem \ref{mainthm} implies the uniqueness of a solution.
\end{proof}
	\section{Proof of Theorem \ref{mainthm}}
	\subsection{Measure Change of FBSDE}
	In this subsection, we provide sufficient conditions that guarantee the existence of a strong solution under the Girsanov transform. We neither assume the non-degeneracy of $\sigma$ nor the boundedness of $|b(t,0)|+\sup_{|x-x'|\leq 1}|b(t,x)-b(t,x')|$. Instead, we assume the following conditions:
	\begin{itemize}
		\item[(H1)] The SDE
		\begin{align}\label{decoupledfsde}
			dF_t=b(t, F_t)dt+\sigma(t,F_t)dW_t
		\end{align}
		has a unique strong solution.
		\item[(H2)]For the strong solution $F$ obtained in {\rm (H1)}, there exist Borel measurable functions $(u,d):[0,T]\times\bbR^m\to\bbR^d\times\bbR^{d\times n}$ such that $U_t=u(t,F_t)$ and $V_t=d(t,F_t)$ is a strong solution of BSDE
		\begin{align}\label{decoupledbsde}
			dU_t&=-f(t,F_t,U_t,V_t)-V_tg(t,F_t,U_t,V_t)dt+V_tdW_t;  &U_T&=h(F_T).
		\end{align}
		\item[(H3)] For $(F,U,V)$ in {\rm (H1)} and {\rm (H2)}, the process
		\begin{align*}
			\scE\brak{\int_0^\cdot g(s,F_s, U_s,V_s)^\intercal d W_s}
		\end{align*}
		is a martingale on $[0,T]$.
		\item[(H4)] For $u,d$ in (H3), the forward SDE
		\begin{align*}
			d\tilde F_t&=\brak{b(t,\tilde F_t)+\sigma(t,\tilde F_t)g(t,\tilde F_t, u(t,\tilde F_t),d(t,\tilde F_t))}dt+\sigma(t,\tilde F_t)dW_t; &\tilde F_0&=x
		\end{align*}
		has a (pathwise) unique strong solution $\tilde F$.
	\end{itemize}
	\begin{lemma}\label{girsanovlemma}
		Assume  {\rm (H1)--(H4)}. Then, the FBSDE
		\begin{align}\label{fbsde1}
			dX_t&=\brak{b(t,X_t)+\sigma(t,X_t) g(t,X_t,Y_t,Z_t)}dt+\sigma(t,X_t) dW_t; &X_0&=x\\\label{fbsde2}
			dY_t&=-f(t,X_t,Y_t,Z_t)dt+Z_tdW_t;&Y_T&=h(X_T).
		\end{align}
		has a strong solution $(X,Y,Z)$ such that $(Y,Z)= (u(t,X_t),d(t,X_t))$. In addition, if BSDE \eqref{decoupledbsde} has a unique strong solution, then \eqref{fbsde1}--\eqref{fbsde2} has a unique strong solution $(X,Y,Z)$ such that $\scE\brak{-\int_0^\cdot g(s,X_s, Y_s,Z_s)^\intercal d W_s}$ is a martingale on $[0,T]$ .
	\end{lemma}
	\begin{proof}
		By (H3), if we define
		\[
		B_t=W_t-\int_0^tg(s,F_s,U_s,V_s)ds,
		\]
		then $B$ is a $\tilde\bbP$-Brownian motion, where $\left.\frac{d\tilde\bbP}{d\bbP}\right|_t=\scE_t\brak{\int_0^\cdot g(s,F_s, U_s,V_s)^\intercal d W_s}$. Under $\tilde\bbP$, the FBSDE \eqref{decoupledfsde}--\eqref{decoupledbsde} becomes
		\begin{align*}
			dF_t&=\brak{b(t,F_t)+\sigma(t,F_t)g(t,F_t,u(t,F_t),d(t,F_t))}dt+\sigma(t,F_t)dB_t; & F_0&=x\\
			dU_t&=-f(t,F_t,U_t,V_t)dt+V_tdB_t;  &Y_T&=h(F_T)
		\end{align*}
		by (H1) and (H2). Note that $F$ is a strong solution by the pathwise uniqueness assumption on (H4). Therefore, $F$ is adapted to the augmented filtration generated by $B$, and so do $(U_t=u(t,F_t):t\in[0,T])$ and $(V_t=d(t,F_t):t\in[0,T])$. As a result, $(F,U,V)$ solves the FBSDE \eqref{fbsde1}--\eqref{fbsde2} and is adapted to the filtration generated by the underlying Brownian motion $B$. This implies that, for the unique strong solution $X$ of the SDE in (H4), $(X,u(\cdot,X.),d(\cdot,X.))$ is a strong solution of \eqref{fbsde1}--\eqref{fbsde2}.
		
		On the other hand, let $(X,Y,Z)$ and $(\tilde X,\tilde Y,\tilde Z)$ be strong solutions of \eqref{fbsde1}--\eqref{fbsde2} such that 
		\begin{align*}
			\scE\brak{-\int_0^\cdot g(s,X_s, Y_s,Z_s)^\intercal d W_s}\text{ and } \scE\brak{-\int_0^\cdot g(s,\tilde X_s, \tilde Y_s,\tilde Z_s)^\intercal d W_s}
		\end{align*}	
		are martingales on $[0,T]$.
		Then, by the Girsanov transform, for
		\begin{align*}
			B_t=W_t+\int_0^tg(s,X_s,Y_s,Z_s)ds\qquad\text{and}\qquad
			\tilde B_t=W_t+\int_0^tg(s,\tilde X_s,\tilde Y_s,\tilde Z_s)ds,
		\end{align*}
		both $(X,Y,Z,B)$ and $(\tilde X,\tilde Y,\tilde Z,\tilde B)$ are weak solutions of \eqref{decoupledfsde}--\eqref{decoupledbsde}. As the \eqref{decoupledfsde} enjoys the pathwise uniqueness, $X=\tilde X$  almost surely for a given Brownian motion $W$. As the BSDE \eqref{decoupledbsde} has a unique solution, we obtain $(Y,Z)=(\tilde Y,\tilde Z)$ almost surely for a given Brownian motion $W$.
	\end{proof}
	\begin{Remark}
		It is also possible to construct a solution of the decoupled FBSDE using a solution of coupled FBSDE. This technique can be used to study multidimensional quadratic BSDE (see Section 2 of \cite{cheridito2015multidimensional}).
	\end{Remark}
	\begin{Remark}
		Lemma \ref{girsanovlemma} can be extended to the case where $\sigma$ also depends on $Y,Z$ as well. In this case, the transformed FBSDE 
		\begin{align*}
			dX_t&=b(t,X_t)dt+\sigma(t,X_t, Y_t,Z_t) dW_t; &X_0&=x\\
			dY_t&=-\brak{f(t,X_t,Y_t,Z_t)+Z_tg(t,X_t,Y_t,Z_t)}dt+Z_tdW_t;&Y_T&=h(X_T)
		\end{align*}
		still has coupling through $\sigma$; therefore, it does not simplify the problem. As a result, we cannot obtain the existence results such as Theorem \ref{mainthm}.
	\end{Remark}
	\subsection{Verification of (H1)}
	In this subsection, we prove that (H1) is satisfied under either (F1), (F2), or (F3). In addition, we obtain the Markovian representation for the solution $F$ of \eqref{decoupledfsde}, which will be used in the subsequent subsections.
	Let us use the following definition introduced by \cite{hamadene1997bsdes}.
	\begin{definition}\label{def:L2dom}
		Consider a class of SDEs
		\begin{align}\label{L2SDE}
			dX^{(t,x)}_s=b(s, X^{(t,x)}_s)ds+\sigma(s, X^{(t,x)}_s)dW_s; \qquad X^{(t,x)}_t=x\in\bbR^m
		\end{align}
		defined on $[t,T]$. We say that the coefficients $(b,\sigma)$ satisfy the $L^2$-domination condition if the following conditions are satisfied:
		\begin{itemize}
			\item For each $(t,x)\in[0,T]\times\bbR^m$, the SDE \eqref{L2SDE} has a unique strong solution $X^{(t,x)}$ for any $(t,x)\in[0,T]\times\bbR^m$. We denote $\mu^{(t,x)}_s$ as the law of $X^{(t,x)}_s$, that is, $\mu^{(t,x)}_s:=\bbP\circ (X^{(t,x)}_s)^{-1}$. 
			\item For any $t\in[0,T], a\in\bbR^m, \mu^{(0,a)}_t$-almost every $x\in\bbR^m$, and $\delta\in(0,T-t]$, there exists a function $\phi_t:[t,T]\times\bbR^m\to\bbR_+$ such that
			\begin{itemize}
				\item for all $k\geq 1$, $\phi_t\in L^2([t+\delta,T]\times[-k,k]^m;\mu^{(0,a)}_s(d\xi)ds)$
				\item $\mu^{(t,x)}_s(d\xi)ds=\phi_t(s,\xi)\mu^{(0,a)}_s(d\xi)ds$
			\end{itemize}
		\end{itemize}
	\end{definition}
	\begin{proposition}\label{L2dominance}
		If (H1) holds, $(b,\sigma)$ satisfies the $L^2$-domination condition and $\bbE|F_t|^2$ is bounded uniformly for $t\in[0,T]$. 
	\end{proposition}
	\begin{proof}
			First, let us assume $b$ is bounded as in (F1) or (F2). Since $\sigma\sigma^\intercal$ is bounded, by Theorem 1 of \cite{aronson1967bounds}, there are constants $K$ and $\lambda$ which only depends on $m, T,\veps$ and $C$ that satisfies
			\[
			K^{-1}(s-t_0)^{-m/2}\exp\brak{-\frac{\lambda^{-1}|\xi-x_0|^2}{s-t_0}}\leq \frac{d\mu^{(t_0,x_0)}_s}{d \xi}\leq K(s-t_0)^{-m/2}\exp\brak{-\frac{\lambda|\xi-x_0|^2}{s-t_0}}
			\]
			for any $(t_0,x_0), (s,\xi)\in(0,T)\times\bbR^m$ with $s>t_0$. Note that
				$\mu^{(t,x)}_s(d\xi)
				=\phi_t(s,\xi)\mu^{(0,a)}_s(d\xi)$
			where 
			\[\phi_t(s,\xi):=\frac{d\mu^{(t,x)}_s}{d \xi}\brak{\frac{d\mu^{(0,a)}_s}{d \xi}}^{-1}\leq K^2\brak{\frac{s}{s-t}}^{m/2}\exp\brak{-\frac{\lambda|\xi-x|^2}{s-t}+\frac{\lambda^{-1}|\xi-a|^2}{s}},
			\]
			$\phi_t\in L^2([t+\delta,T]\times[-k,k]^m;\mu^{(0,a)}_s(d\xi)ds)$ for all $k\geq 1$. Therefore, the $L^2$-domination condition holds.
			
			Now, let us consider the case (F3).	Let us denote $\vartheta^{(t,x)}(s)$ to be the solution of ODE
		\[
		\frac{d\vartheta^{(t,x)}}{ds}=\tilde b(s,\vartheta^{(t,x)}(s));\qquad \vartheta^{(t,x)}(t)=x.
		\]
		Here, $\tilde b(s,x)$ is a mollification of $b(s,x)$ with respect to $x$, that is,
		\[
		\tilde b(s,x)=\int_{\bbR^m} b(s,y)p(x-y)dy
		\]
		for some $p\in C^\infty_c(\bbR^m;\bbR)$ with support in the unit ball and scuh that $\int_{\bbR^m}p(x)dx=1$.
		Then, by Theorem 1.2 of \cite{menozzi2021density}, there are constants $\lambda\in(0,1]$ and $ K\geq 1$, which may depend on $T,\alpha, \veps, \kappa$ such that
		\small
		\[
		K^{-1}(s-t_0)^{-m/2}\exp\brak{-\frac{\lambda^{-1}|\vartheta^{(t_0,x_0)}(s)-\xi|^2}{s-t_0}}\leq\frac{d\mu^{(t_0,x_0)}_s}{d \xi}\leq K(s-t_0)^{-m/2}\exp\brak{-\frac{\lambda|\vartheta^{(t_0,x_0)}(s)-\xi|^2}{s-t_0}}
		\]
		\normalsize
		for any $(t_0,x_0), (s,\xi)\in(t_0,T)\times\bbR^m$.
		Note that $\mu^{(t,x)}_s(d\xi)
		=\phi_t(s,\xi)\mu^{(0,a)}_s(d\xi)$ where 
		\[\phi_t(s,\xi):=\frac{d\mu^{(t,x)}_s}{d \xi}\brak{\frac{d\mu^{(0,a)}_s}{d \xi}}^{-1}\leq K^2\brak{\frac{s}{s-t}}^{m/2}\exp\brak{-\frac{\lambda|\vartheta^{(t,x)}(s)-\xi|^2}{s-t}+\frac{\lambda^{-1}|\vartheta^{(0,a)}(s)-\xi|^2}{s}}.
		\]
		As $\vartheta^{(t,x)}(s)$ is uniformly bounded for $s\in[t,T]$,  $\phi_t$ is bounded on $[t+\delta,T]\times[-k,k]^m$ for any $k\geq 1$. Therefore, the $L^2$-domination condition holds.
		
		On the other hand, note that there exists a constant $K$ such that $|b(t,x)|\leq K(1+|x|)$. Then, there exist non-negative constants $K_1$ and $K_2$ that satisfy
		\begin{align*}
			\bbE|F_t|^2&\leq K_1\brak{|F_0|^2+\bbE\int_0^t|b(s,F_s)|^2ds+\bbE\int_0^T|\sigma(s,F_s)|^2ds}\leq K_2\brak{1+\int_0^t\bbE|F_s|^2ds}
		\end{align*}
		By Gr\"onwall's inequality, we have $\sup_{t\in[0,T]}\bbE|F_t|^2<\infty$.
	\end{proof}
	\begin{proposition}\label{H1sufficient}
		If either {\rm (F1), (F2)}, or {\rm (F3)} holds, then \eqref{decoupledfsde} has a unique strong solution $F$, $\bbE|F_t|^2$ is bounded uniformly for $t\in[0,T]$, and $(b,\sigma)$ satisfies the $L^2$-domination condition.
	\end{proposition}
	\begin{proof}
		It is easy to verify that \eqref{decoupledfsde} has a unique strong solution if either {\rm (F1) or (F2)} holds (see \cite{gyongy2001stochastic} and \cite{le1984one}). 
		
		On the other hand, assume that (F3) holds. As a symmetric matrix $A:=\sigma\sigma^\intercal$ has a nonzero determinant, $A$ has orthonormal eigenvectors $\crl{\xi_1,\cdots,\xi_m}$ with the corresponding strictly positive eigenvalues $\crl{\lambda_1,\cdots,\lambda_m}$. Let
		\[
		E:=\brak{\frac{1}{\sqrt{\lambda_j}}\xi_j:j=1,2,\dotsc, m }.
		\]
		Then, $E^\intercal AE$ becomes an identity matrix. Therefore, by L\'evy characterization, we know that $B:=U^\intercal \sigma W$ is a $\bbP$-Brownian motion. Note that \cite{menoukeu2019flows} shows that the forward SDE
		\begin{align*}
			dP_t=E^\intercal b(t,\brak{E^\intercal}^{-1} P_t)dt+ dB_t;\qquad P_0=E^\intercal x
		\end{align*}
		has a unique strong solution because $b(t,x)$ has a linear growth in $x$, as pointed out in Remark \ref{lineargrowth}. Therefore, $\tilde X:= \brak{E^\intercal}^{-1} P$ is a unique solution of \eqref{decoupledfsde}. 
		
		Therefore, (H1) is satisfied under one of the assumptions (F1), (F2), or (F3). The remainder of our claim is proved by Proposition \ref{L2dominance}.
	\end{proof}
	\subsection{Verification of (H2) and (H3)}
	In this subsection, we will always assume (H1).
	\begin{proposition}\label{B1H23}
		{\rm (B1)} implies {\rm (H2)} and {\rm (H3)}.
	\end{proposition}
	\begin{proof} First, let us assume $r>0$, which implies $\rho_r\equiv 0$. Then, (H3) holds automatically because $g$ is bounded. Note that $\tilde X$ obtained by \eqref{decoupledfsde} is a Markov process because the corresponding Martingale problem is well posed. By Propositions \ref{L2dominance} and \ref{decoupledLip}, all of the conditions in Remark 27.3 of \cite{hamadene1997bsdes} are satisfied. As $\bar f$ exhibits linear growth in $(y,z)$, Theorem 27.2 of \cite{hamadene1997bsdes} proves (H2). 
		
		On the other hand, if $r=0$, then $h$ and $f(t,x,0,0)$ are bounded by $C$. Then, we will show that $U_t=u(t,F_t)$ is uniformly bounded by $e^{\half aT}\sqrt{C^2+T}$. If so, (H3) is automatically satisfied. Moreover, because $\bar f$ exhibits linear growth in $(y,z)$, (H2) holds by the previous argument again.
		
		For a positive constant $N$, let
		\[
		P_N:[0,T]\times\bbR^m\times\bbR^d\times\bbR^{d\times n}\ni(t,x,y,z)\mapsto \brak{t,x,\frac{Ny}{|y|\vee N},z}\in[0,T]\times\bbR^m\times\bbR^d\times\bbR^{d\times n}
		\]
		and $f_N:=f\circ P_N, g_N:=g\circ P_N$, and $H_N(t,x,y,z):=f_N(t,x,y,z)+zg_N(t,x,y,z)$. Then, by (B1), there exists a constant $\tilde C>0$ such that
		\[
		|H_N(t,x,y,z)|\leq \tilde C(1+|z|).
		\]
		Then, by the same argument for $r>0$, the FBSDE
		\begin{equation}\label{app:lem2}
			\begin{aligned}
				dF_t&=b(t,F_t)dt+\sigma dW_t; & F_0&=x\\
				dU_t&=-H_N(t,F_t,U_t,V_t)dt+V_tdW_t;  &U_T&=h(F_T)
			\end{aligned}
		\end{equation}
		has a strong solution $(F, U,V)$ such that there exist Borel measurable functions $(u,d):[0,T]\times\bbR^m\to\bbR^d\times\bbR^{d\times n}$ such that $U_t=u(t,F_t)$ and $V_t=d(t,F_t)$ $dt\otimes d\bbP$-almost everywhere.
		By It\^o formula, we have
		\begin{align*}
			e^{at}|U_t|^2
			&=e^{aT}|h(F_T)|^2+\int_t^Te^{as}\brak{2U_s^\intercal f_N(s,F_s,U_s,V_s)-|V_s|^2- a|U_s|^2} ds-\int_t^T2e^{as}U_s^\intercal V_sd\tilde W_s\\
		\end{align*}
		where $a=2C(C+1)$ and \[
		\tilde W_t=W_t - \int_0^tg_N(s,F_s,U_s,V_s) ds.
		\]
		Note that $\tilde W$ is a Brownian motion under some measure $\tilde \bbP$ because $g_N$ is bounded. In addition, by using the inequality $2Cxy\leq C^2x^2+y^2$, we have
		\begin{align*}
			&2U_t^\intercal f_N(t,F_t,U_t,V_t)-|V_t|^2- a|U_t|^2\\
			&\leq 2C|U_t|(1+|U_t|+|V_t|)-|V_t|^2- a|U_t|^2\leq  (2C^2+2C-a)|U_t|^2+1\leq 1.
		\end{align*}
		Therefore, if we denote $\tilde\bbE$ as the expectation with respect to $\tilde\bbP$, we obtain
		\begin{align*}
			|U_t|^2&=e^{-at}\tilde \bbE\edg{e^{aT}|h(F_T)|^2+\int_t^Te^{as} \brak{2U_s^\intercal f_N(s,F_s,U_s,V_s)-|V_s|^2- a|U_s|^2} ds}\\
			&\leq e^{a(T-t)}\brak{C^2+T-t}.
		\end{align*}
		Therefore, $U$ is uniformly bounded, independent of the choice of $N$. If we set $N\geq e^{\half aT}\sqrt{C^2+T}$, the solution of \eqref{app:lem2} is the solution of \eqref{decoupledbsde} and $|U_t|\leq N$. This proves the claim.
	\end{proof}
	\begin{proposition}\label{B2H23}
		{\rm (B2)} implies {\rm (H2)} and {\rm (H3)}.
	\end{proposition}
	\begin{proof}
		Note that by \cite{hu2016multi}, BSDE \eqref{decoupledbsde} has a unique solution $(U, V)$ such that $U$ is bounded. In this case, $V\cdot W$ is a BMO martingale. Therefore, without loss of generality, we can assume that $\hat f$ is a bounded Lipschitz function and $|g(t,x,y,z)|\leq C$. Then, (H3) is satisfied.
		
		Now, we only need to prove the existence of measurable functions $u$ and $d$ such that $(U_t, V_t)=(u(t,F_t),d(t,F_t))$. Let $U^{(0)}_t=u_0(t, F_t)= 0$ for all $t\in[0,T]$ and $V^{(0)}_t=d_0(t,F_t)= 0$, and we define
		\begin{align*}
			U^{(k+1),i}_t=h^i(F_T)+\int_t^T\hat f^i(s,F_s,U^{(k)}_s, V^{(k)}_s)+\tilde f^i(s,F_s,V^{(k+1),i}_s)ds-\int_t^TV^{(k+1),i}dW_t.
		\end{align*}
		Then, as shown in the proof of Proposition \ref{B3H23}, there exist measurable functions $u_k:[0,T]\times\bbR^m\to\bbR^d$ and $d_k:[0,T]\times\bbR^m\to\bbR^{d\times n}$ such that $U^{(k)}_t=u_k(t,F_t)$ and $V^{(k)}_t=d_k(t, F_t)$.
		As $U^{(k)}\to U$ in $\bbS^\infty$ and $V^{(k)}\cdot W\to V\cdot W$ in BMO by \cite{hu2016multi}, if we let \(u^i(t,x):=\limsup_{k\to\infty}u_k^i(t,x)\) and \(d^{ij}(t,x):=\limsup_{k\to\infty}d_k^{ij}(t,x)\),
		where $u=(u^i)_{1\leq i\leq d}$ and $d=(d^{ij})_{1\leq i\leq d,1\leq j\leq n}$, we have
		\begin{align*}
			u^i(t,F_t)&=(\limsup_{k\to\infty}u_k^i)(t,F_t)=\limsup_{k\to\infty}(u_k^i(t,F_t))=\lim_{k\to\infty}U^{(k),i}_t=U^i_t\\
			d^{ij}(t,F_t)&=(\limsup_{k\to\infty}d^{ij}_k)(t,F_t)=\limsup_{k\to\infty}(d^{ij}_k(t,F_t))=\lim_{k\to\infty}V^{(k),ij}_t=V^{ij}_t
		\end{align*}
		in $dt\otimes d\bbP$-everywhere sense.
		Therefore, (H2) holds.
	\end{proof}
	\begin{proposition}\label{B3H23}
		{\rm (B3)} implies {\rm (H2)} and {\rm (H3)}.
	\end{proposition}
	\begin{proof}
		The existence of solution $(F,U,V)$ for BSDE \eqref{decoupledbsde} can be seen in \cite{kobylanski2000backward}. In particular, $U$ is bounded and $V\cdot W$ is a BMO martingale (see \cite{briand2013simple}). Therefore, without loss of generality, we assume that
		\begin{align*}
			|f(t,x,y,z)|\leq C(1+|z|^2)\qquad\text{and}\qquad
			|g(t,x,y,z)|\leq C,
		\end{align*}
		and therefore, (H3) holds.
		
		On the other hand, by \cite{kobylanski2000backward}, there is a sequence of measurable functions $\theta_k(t,x,y,z)$ such that
		\begin{itemize}
			\item $\theta_k(t,x,y,z)$ is uniformly Lipschitz in $(y,z)$
			\item For the solution $(Y^{(k)}, Z^{(k)})$ of BSDE
			\begin{align*}
				Y^{(k)}_t=\exp(Lh(F_t))+\int_t^T \theta_k(s, F_s, Y^{(k)}_s, Z^{(k)}_s)-\int_t^TZ^{(k)}_sdW_s,
			\end{align*}
			we have
			\begin{align*}
				\lim_{k\to\infty} \frac{\log(Y^{(k)}_t)}{2L}&= U_t\text{ uniformly in }t\\
				\lim_{k\to\infty} \frac{Z^{(k)}}{2LY^{(k)}}&= V\text{ in } \bbH^2.
			\end{align*}
		\end{itemize}
		Here, $L$ is a constant determined by coefficients $h$ and $\bar f$.
		
		From Proposition \ref{decoupledLip}, there are measurable functions $u_k:[0,T]\times\bbR^m\to\bbR$ and $d_k:[0,T]\times\bbR^m\to\bbR^{1\times n}$ such that $Y^{(k)}_t=u_k(t,F_t)$ and $Z^{(k)}_t=d_k(t, F_t)$. Therefore, if we let \(u(t,x):=\limsup_{k\to\infty}u_k(t,x)\) and \(d^{i}(t,x):=\limsup_{k\to\infty}d_k^{i}(t,x)\), where $d=(d^i)_{1\leq i\leq n}$, then $U_t=u(t,F_t)$ and $V_t=d(t, F_t)$. This proves (H2).
	\end{proof}
	\begin{proposition}\label{B4H23}
		{\rm (F3)} and {\rm (B4)} imply {\rm (H2)} and {\rm (H3)}.
	\end{proposition}
	\begin{proof}
		For (H2), the proof is identical to that for Theorem 3.1 of \cite{mu2015one}. Let $E$ be the matrix defined in the proof of Proposition \ref{H1sufficient}. Then,
		\begin{align*}
			dP_t=E^\intercal b(t,\brak{E^\intercal}^{-1} P_t)dt+ E^\intercal\sigma dW_t;\qquad P_0=E^\intercal x
		\end{align*}
		for $P=E^\intercal F$.
		Note that there exists $ C'>0$ such that
		\begin{align*}
			|P_t|&\leq |{E^\intercal}x|+\int_0^t|E^\intercal b(s,\brak{E^\intercal}^{-1}P_s)|ds+|E^\intercal\sigma W_t|\\
			&\leq C'+|E^\intercal\sigma W_t|+C'\int_0^t|P_s|ds
		\end{align*}
		By Gronwall's inequality, there exists a constant $C''>0$ such that
		\[
		|P_t|\leq C''\left(1+\max_{s\in[0,t]}|E^\intercal\sigma W_s|\right).
		\]
		This implies
		\begin{align*}
			|F_t|= |(E^\intercal)^{-1} P_t|\leq C'''\left(1+\max_{s\in[0,t]}|W_s|\right)
		\end{align*}
		for a constant $C'''>0$. Therefore, by the Bene\v{s} condition (see Corollary III.5.16 of \cite{karatzas1998brownian}), (H3) holds.
	\end{proof}
	\subsection{Verification of (H4)}
	Let us prove (H4) under the assumptions made in Theorem \ref{mainthm}.
	\begin{proof}
		Note that if {\rm (H1)} and either one of {\rm (B1), (B2)} or {\rm (B3)} hold, $g(t, F_t, U_t,V_t)$ is bounded because $U_t$ is bounded. Therefore, we only need to prove that (H4) holds for $g_N(t,x,y,z):=g(t,x,Ny/\brak{|y|\vee N},z)$. If either (F1) or (F2) holds, then $\tilde b(t,x):=b(t,x)+\sigma(t,x)g_N(t,x,u(t,x),d(t,x))$ is bounded because $\sigma$ and $g_N$ are bounded. Therefore, conditions (F1) or (F2) hold with $\tilde b$ instead of $b$. Likewise, if (F3) holds, then there exists a positive constant $\kappa$ such that, 
		$$
		|\tilde b(t,0)|+\sup_{|x-x'|\leq 1}|\tilde b(t,x)-\tilde b(t,x')|\leq \kappa
		$$
		for all $t\in[0,T], x,x'\in\bbR^m$.
		By Proposition \ref{H1sufficient}, (H4) holds.
		
		On the other hand, assume (F3) and (B4) hold. Note that, if $r=0$ in (B4), then $|U|$ is bounded; therefore, $|g(t,x,y,z)|\leq K(1+|x|)$ for some $K$. If $r>0$, $|g(t,x,y,z)|\leq K(1+|x|)$ for $K=C$. Therefore, we have
		\[
		|\tilde b(t,x)|\leq |b(t,x)|+|\sigma(t,x)||g(t,x,y,z)|\leq C'(1+|x|)
		\]
		for a non-negative constant $C'$. Again, by Proposition \ref{H1sufficient}, (H4) holds.
	\end{proof}
	\subsection{Uniqueness}
	Assume the conditions in Theorem \ref{mainthm}.
	Let us prove that either (B2), (U1), or (U2) implies the uniqueness of the solution for \eqref{fbsde1}--\eqref{fbsde2}.
	\begin{proof}
		When $f(t,x,y,z)+zg(t,x,y,z)$ is Lipschitz, the solution for \eqref{decoupledbsde} is unique by \cite{pardoux1990adapted}. If $r=0$ and $d=1$, \cite{kobylanski2000backward} proved the uniqueness of the solution for \eqref{decoupledbsde}. The uniqueness of the solution under the condition (B2) was proved by \cite{hu2016multi}. By applying Lemma \ref{girsanovlemma}, the uniqueness of the solution for \eqref{fbsde1}--\eqref{fbsde2} is proved.
	\end{proof}

	\bibliographystyle{abbrvnat}
	\bibliography{mybib}
	
	\appendix
	\section{Markovian Solution of Decoupled FBSDE}
	In this appendix, we provide sufficient conditions for (H1) and (H3) under the assumption that the forward SDE \eqref{decoupledfsde} has a unique strong solution.
	
	It is well known that a decoupled FBSDE with a Lipschitz BSDE driver has a unique solution if the forward SDE has a unique strong solution. The Markovian property of the solution has been proved in Theorem 4.1 of \cite{el1997backward} and Theorem 14.5 of \cite{barles1997sde} under the assumption that the forward SDE has Lipschitz coefficients. The following theorem slightly generalizes the existence and uniqueness results in the sense that we do not require $b(t,x)$ and $\sigma(t,x)$ to be Lipschitz with respect to $x$ and we allow linear growth of $\bar f(t,x,y,z)$ with respect to $(y,z)$.
	\begin{proposition}\label{decoupledLip}
		Let $\bar f(t,x,y,z):=f(t,x,y,z)+zg(t,x,y,z)$ for jointly $\scB$-measurable functions $(f,g):[0,T]\times\bbR^m\times\bbR^d\times\bbR^{d\times n}\to\bbR^d\times\bbR^{n}$.
		Assume the following conditions: there exist $C>0, p\geq 2,$ and $r\geq \half$ such that
		\begin{itemize}
			\item The forward SDE \eqref{decoupledfsde} has a unique strong solution $F$ and 
			$\bbE \sup_{t\in[0,T]}|F_t|^{pr}\leq C$
			\item \(|h(x)|\leq C(1+|x|^r)\) for all $x\in\bbR^m$
			\item \(\abs{\bar f(t,x,y,z)}\leq C(1+|x|^r+|y|+|z|)\) for all $(t,x,y,z)\in[0,T]\times\bbR^m\times\bbR^d\times\bbR^{d\times n}$
			\item \(\abs{\bar f(t,x,y,z)-\bar f(t,x,y',z')}\leq C\brak{|y-y'|+|z-z'|}\) for all \((t,x,y,z),(t,x,y',z')\in[0,T]\times\bbR^m\times\bbR^d\times\bbR^{d\times n}\)
		\end{itemize}
		Then, (H2) holds. Moreover, the solution $(F,U,V)\in\bbS^{pr}(\bbR^m)\times\bbS^{p}(\bbR^{d})\times\bbH^{p}(\bbR^{d\times n})$ is unique.
	\end{proposition}
	\begin{Remark}
		We do not need the nondegeneracy of $\sigma$ in this proposition.
	\end{Remark}
	\begin{proof}
		Note that, \(\bbE|g(F_T)|^p\leq C^p\bbE(1+|F_T|^r)^p\leq 2^{p-1}C^p\brak{1+\bbE|F_T|^{pr}}<\infty\) and
		\begin{align*}
			\bbE \edg{\brak{\int_0^T|\bar f(t,F_t,0,0)|^2dt}^{p/2}}
			&\leq C^p\bbE \edg{\brak{\int_0^T2\brak{1+|F_t|^{2r}}dt}^{p/2}}\\
			&\leq 2^{p/2}C^pT^{p/2}\brak{1+\bbE\sup_{t\in[0,T]}|F_t|^{pr}}<\infty.
		\end{align*}
		Therefore, from the classical result of \cite{pardoux1990adapted}, the BSDE \eqref{decoupledbsde} has a unique solution $(Y,Z)\in\bbS^p(\bbR^d)\times\bbH^p(\bbR^{d\times n})$.
		
		For \(k=0,1,2,...\), let us define Borel measurable functions \((u_k,d_k):[0,T]\times\bbR^m\ni(t,x)\mapsto (u_k(t,x),d_k(t,x))\in\bbR^d\times\bbR^{d\times n}\) as follows: let $u_0\equiv0,d_0\equiv 0$, \(Y^{(k)}_t=u_k(t,F_t)\) and \(Z^{(k)}_t=d_k(t,F_t)\) \(dt\otimes d\bbP\)-almost everywhere for \((Y^{(k)}, Z^{(k)})\), which is the unique solution of
		\begin{align*}
			dY^{(k)}_t:=-\bar f(t,F_t,u_{k-1}(t,F_t), d_{k-1}(t,F_t))dt+Z^{(k)}_tdB_t;\qquad Y^{(k)}_T=g(F_T).
		\end{align*}
		The well-definedness of $(u_k,d_k)_{k=0,1,2,...}$ is proved in Lemma \ref{lem:ukdk}.It is well known that $Y^{(k)}\to Y$ in $\bbS^p$ and $Z^{(k)}\to Z$ in $\bbH^p$.
		If we let \(u^i(t,x):=\limsup_{k\to\infty}u_k^i(t,x)\) and \(d^{ij}(t,x):=\limsup_{k\to\infty}d_k^{ij}(t,x)\),
		where $u=(u^i)_{1\leq i\leq d}$ and $d=(d^{ij})_{1\leq i\leq d,1\leq j\leq n}$, we have
		\begin{align*}
			u^i(t,P_t)&=(\limsup_{k\to\infty}u_k^i)(t,P_t)=\limsup_{k\to\infty}(u_k^i(t,P_t))=\lim_{k\to\infty}Y^{(k),i}_t=Y^i_t\\
			d^{ij}(t,P_t)&=(\limsup_{k\to\infty}d^{ij}_k)(t,P_t)=\limsup_{k\to\infty}(d^{ij}_k(t,P_t))=\lim_{k\to\infty}Z^{(k),ij}_t=Z^{ij}_t.
		\end{align*}
		Therefore, the claim is proved.
	\end{proof}
	
	Now, let us prove that $(u_k,d_k)_{k=0,1,2,...}$ are well defined. 
	\begin{lemma}\label{lem:ukdk}
		For all $k=0,1,2...$, we have that $(u_k,d_k)$ are well-defined. Moreover, $u_k(\cdot, F.)\in\bbS^p(\bbR^d)$ and $d_k(\cdot, F.)\in\bbH^p(\bbR^{d\times n})$.
	\end{lemma}
	\begin{proof}
		We prove this by mathematical induction. First, note that the claim holds true for $k=0$. Assume that the claim holds for $k-1\geq 0$. It should be noted that \(\bbE|g(F_T)|^p\leq C^p\bbE(1+|F_T|^r)^p\leq 2^{p-1}C^p\brak{1+\bbE|F_T|^{pr}}<\infty\) and
		\begin{align*}
			&\bbE \edg{\brak{\int_0^T|\bar f(t,F_t,u_{k-1}(s,F_s), d_{k-1}(s,F_s))|^2dt}^{p/2}}\\
			&\leq C^p\bbE \edg{\brak{\int_0^T4\brak{1+|F_t|^{2r}+|u_{k-1}(s,F_s)|^2+|d_{k-1}(s,F_s)|^2}dt}^{p/2}}\\
			&\leq 2^{p}C^pT^{p/2}\brak{1+\bbE\sup_{t\in[0,T]}|F_t|^{pr}+\bbE\sup_{t\in[0,T]}|u_{k-1}(t, F_t)|^{p}+\bbE\edg{\brak{\int_0^T|d_{k-1}(t,F_t)|^2dt}^{p/2}}}<\infty.
		\end{align*}
		Therefore, the BSDE
		\begin{align*}
			Y^{(k)}_t=g(F_T)+\int_t^T\bar f(s,F_s,u_{k-1}(s,F_s), d_{k-1}(s,F_s))ds-\int_t^TZ^{(k)}_sdB_s
		\end{align*}
		has a unique solution such that $Y^{(k)}\in\bbS^p(\bbR^d)$ and $Z^{(k)}\in\bbH^p(\bbR^{d\times n})$. Note that, because $(t,F_t)_{t\geq0}$ is a Markov process, we know
		\begin{align*}
			Y^{(k)}_t&=\bbE\edg{\left.g(F_T)+\int_t^T\bar f(s,F_s,u_{k-1}(s,F_s), d_{k-1}(s,F_s))ds\right|\scF_t}\\
			&=\bbE\edg{\left.g(F_T)+\int_t^T\bar f(s,F_s,u_{k-1}(s,F_s), d_{k-1}(s,F_s))ds\right|F_t}.
		\end{align*}
		Moreover, by Proposition II.4.6 of \cite{cinlar2011probability}, there exists a Borel measurable function $u_k:[0,T]\times\bbR^m\to\bbR^{d}$ such that \(Y^{(k)}_t=u_k(t,F_t).\)
		On the other hand, note that \[Y^{(k)}_t+\int_0^t\bar f(s,F_s,u_{k-1}(s,F_s), d_{k-1}(s,F_s))ds\] is an additive martingale. By Theorem 6.27 of \cite{cinlar1980semimartingales}, there exists a $\bar\scB$-measurable function $\bar d_k:[0,T]\times\bbR^m\to\bbR^{d\times n}$ such that
		\begin{align*}
			Y^{(k)}_t+\int_0^t\bar f(s,F_s,u_{k-1}(s,F_s), d_{k-1}(s,F_s))ds=Y^{(k)}_0+\int_0^t\bar d_k(s, F_s)dB_s.
		\end{align*}
		Here $\bar \scB$ is the $\sigma$-algebra of universally measurable sets.
		Let $G(t,\omega):=(t,F_t(\omega))$ and consider $\mu:=(\lambda\otimes\bbP)\circ G^{-1}$, where \(\lambda\) is the Lebesgue measure on \([0,T]\). Then, \(\mu\) is a finite measure on \([0,T]\times\bbR^m\); therefore, there exists a $\scB$-measurable function $d_k:[0,T]\times\bbR^m\to\bbR^{d\times n}$ such that
		\begin{align*}
			\mu\brak{\crl{(t,x)\in[0,T]\times\bbR^m: \bar d_k(t,x)\neq d_k(t,x)}}=0.
		\end{align*}
		This implies that $\bar d_k(t,F_t)=d_k(t,F_t)$ in $dt\otimes d\bbP$-almost everywhere sense.
		Therefore, the claim is proved.
	\end{proof}
	
\end{document}